\documentclass[12pt,a4paper,reqno,twoside]{amsart}

\usepackage[english]{babel}
\usepackage{stmaryrd}
\usepackage{dsfont}
\usepackage[symbol*,ragged]{footmisc}
\usepackage[colorlinks,linkcolor=red,anchorcolor=blue,citecolor=blue,urlcolor=blue]{hyperref}
\usepackage{color,xcolor}

\usepackage{geometry}
\usepackage{amssymb}
\usepackage{amsmath}
\usepackage{mathrsfs}
\usepackage{amsfonts}
\usepackage{epsfig}

\usepackage{amsthm}
\usepackage{amsxtra}
\usepackage{bbding}
\usepackage{epsfig}
\usepackage{graphicx}
\usepackage{latexsym}
\usepackage{mathbbol}
\usepackage{bbold}

\usepackage{pifont}
\usepackage{wasysym}
\usepackage{skull}
\usepackage{float}

\DeclareSymbolFontAlphabet{\mathbb}{AMSb}
\DeclareSymbolFontAlphabet{\mathbbol}{bbold}

\usepackage{amscd}
\usepackage[all]{xy}
\allowdisplaybreaks[4]
\usepackage{setspace}

\theoremstyle{plain}
\newtheorem{theorem}{\normalfont\scshape Theorem}[section]
\newtheorem{proposition}{\normalfont\scshape Proposition}[section]
\newtheorem{lemma}[proposition]{\normalfont\scshape Lemma}
\newtheorem{corollary}[theorem]{\normalfont\scshape Corollary}
\newtheorem*{corollary*}{\normalfont\scshape Corollary}

\theoremstyle{remark}
\newtheorem*{remark*}{\normalfont\scshape Remark}
\newtheorem*{notation}{\normalfont\scshape Notation}

\numberwithin{equation}{section}
\addtocounter{footnote}{1}

\renewcommand{\footnoterule}{
	\kern -3pt
	\hrule width 2.5in height 0.4pt
	\kern 3pt
}

\makeatletter
\@ifundefined{MakeUppercase}{}{}
\makeatother

\begin{document}
	
	\title[  ]
	{ On the Representation of Large Even Integers as the Sum of Eight Primes from Positive Density Sets }
	
	\author[ ]
	{Meng Gao}
	
	\address{Department of Mathematics, China University of Mining and Technology,
		Beijing 100083, People's Republic of China}
	
	\email{meng.gao.math@gmail.com}

	\date{}
	
	\footnotetext[1]{	
		
		\quad\,\,
		{\textbf{MR(2020) Subject Classification}}: 11D85, 11P32, 11P55, 11P70
		
	}

	\begin{abstract}
		Let $ \mathbb{P} $ denote the set of all primes. We have proved that if $ A $ is a subset of $ \mathbb{P} $, and the lower density of $A$ in $ \mathbb{P} $ is larger than $1/2$, then every sufficiently large even integer $n$ can be expressed in the form $ n=p_{1}+\cdots+p_{8} $, where $ p_{1},\ldots,p_{8}\in A $. The constant $ 1/2 $ in this statement is the best possible.	
	\end{abstract}
	
	\maketitle

	\section{Introduction and main result}
        The transference principle was initially  developed by Green \cite{Gre} and has now become an important method in the study of additive number theory. Employing this method,
        Li-Pan \cite{LP} and Shao \cite{Shao} studied the density version of Vinogradov's three primes theorem and obtained remarkable results. For a set $ A\subseteq \mathbb{P} $, define
        \begin{center}
        	$ \delta_{A}=\underline{\delta}_{A}(\mathbb{P}):=\liminf \limits_{N\rightarrow\infty}\dfrac{|A\cap[N]|}{|\mathbb{P}\cap[N]|} , $
        \end{center}
        where $ [N]:=\{1,\ldots,N\} $.
        
        Let $A_{1}$, $A_{2}$, $A_{3}\subseteq \mathbb{P}$. In \cite{LP}, Li and Pan proved that if $ \delta_{A_{1}}+\delta_{A_{2}}+\delta_{A_{3}}>2 $, then for every sufficiently large odd integer $n$,  there are primes $ p_{i}\in A_{i}\ (i=1,2,3)$ such that  $ n=p_{1}+p_{2}+p_{3} $. Inspired by the work of Li and Pan, Shao \cite{Shao} proved that if $ \delta_{A_{1}}>5/8 $, then for every sufficiently large odd integer $n$, there are
        primes $ p_{1},\ p_{2},\ p_{3} \in A_{1} $ such that $ n=p_{1}+p_{2}+p_{3} $. Furthermore, building on Shao's work, Shen \cite{Shen} extended Shao's result to an asymmetric version. It is worth mentioning that recently Alsetri and Shao \cite{AS} used the transference principle to obtain a density version of the almost all binary Goldbach problem. As for the nonlinear Goldbach problem, the author \cite{Gao} extended the Waring-Goldbach problem to a density version. 
        
        In \cite{Shao}, Shao showed that the constant $5/8$ in his result \cite[Theorem 1.1]{Shao} cannot be improved. Therefore, we hope to lower the density threshold by relaxing the requirement on the number of prime variables, thereby further generalizing Shao's result.
        Our result is the following.
        
        \begin{theorem}\label{thm 1.1}
        	Let $A$ be a subset of $\mathbb{P}$ with $\delta_{A}>1/2$. Then for every sufficiently large even integer $n$, there exist $ p_{1},\ldots,p_{8}\in A  $ such that $ n=p_{1}+\cdots+p_{8} $.
        \end{theorem}
        The constant $ 1/2 $ in Theorem \ref{thm 1.1} cannot be improved. In fact, if we take\begin{center}
        	$ A=\{ p\in \mathbb{P}: p\equiv 1\ (\bmod\ 3) \} $,
        \end{center}
        then $ \delta_{A}=1/2 $, but any even integer $ n\equiv 1\ (\bmod\ 3) $ cannot be written as a sum of eight elements from $ A $.
        
         Note that in \cite[Theorem 1.3]{Shao}, the assumption on the density of $A$ in $ \mathbb{P} $ can also be $ \delta_{A}>1/2 $ (see the discussion following \cite[Theorem 1.3]{Shao}), but $A$ also needs to satisfy local conditions. In Theorem \ref{thm 1.1}, we removed the local conditions at the expense of increasing the number of prime variables.
        
        One of the key components in the proof of Theorem \ref{thm 1.1} is the following theorem.
        \begin{theorem}\label{thm 1.2}
        	Let $q$ be a square-free positive integer. Let $A\subseteq \mathbb{Z}_{q}^{\ast}$ with $|A|>\frac{1}{2}\varphi(q)$. Then \begin{center}
        		$A+A+A+A=\{a\in \mathbb{Z}_{q}:a\equiv 0\ (\bmod (2,q)) \}$,
        	\end{center}
        	where $ \varphi $ is the Euler totient function.
        \end{theorem}
        When $ q $ in Theorem \ref{thm 1.2} is odd, we immediately have the following result.
        \begin{corollary}\label{corollary 2}
        	Let $q$ be an odd square-free positive integer. Let $A\subseteq \mathbb{Z}_{q}^{\ast}$ with $|A|>\frac{1}{2}\varphi(q)$. Then \begin{center}
        		$A+A+A+A=\mathbb{Z}_{q}$.
        	\end{center}
        \end{corollary}
         Corollary \ref{corollary 2} can also be viewed as an extension of \cite[Corollary 1.5]{Shao}. We will give the proof of Theorem \ref{thm 1.2} in Section 2, and prove Theorem \ref{thm 1.1} in
        Section 5.
	\begin{notation}
		For a set $ A \subseteq \mathbb{N} $ and $ s\in \mathbb{N} $, we define $ sA=\{a_{1}+\cdots+a_{s}:a_{1},\ldots,a_{s}\in A\} $. We use $ ||x|| $ to denote the distance from $ x $ to the nearest integer. For $ x\in \mathbb{R} $ and $ q\in \mathbb{N} $,  we use notation $ e(x) $ and $ e_{q}(x) $ for $ e^{2\pi ix} $ and $ e^{2\pi ix/q}$ respectively. We will use $ \mathbb{T} $ to denote the torus  $\mathbb{R}/\mathbb{Z}$ , and will identify $ \mathbb{T} $ with the interval $ [0,1) $ throughout this paper. The letter $ p $, with or without subscript, denotes a prime number.
		For a set $ A $, we write $ 1_{A}(x) $ for  its characteristic function. If $ f:B\rightarrow \mathbb{C} $ is a function and $ B_{1} $ is a non-empty finite subset of $ B $, we write $ \mathbb{E}_{x \in B_{1}}f(x) $ for the average value of $ f $ on  $ B_{1} $,  that is to say
		\begin{center}
			$ \mathbb{E}_{x \in B_{1}}f(x)=\dfrac{1}{|B_{1}|}\sum\limits_{x \in B_{1}}f(x) $.
		\end{center}

	\end{notation}
	
	\section{Mean value estimate}
	Let $ n_{0} $ be a sufficiently large even integer.
	Let $ w=\log\log\log n_{0} $,
	\begin{center}
		$W:=\prod\limits_{1<p\leq w}p $ and $ \ N:=\lfloor n_{0}/(4W) \rfloor$.
	\end{center}
	Let $b\in [W]$ be such that $(b,\ W)=1$. For a set $A\subseteq \mathbb{P} $, 
	define functions $ f_{b},\ \nu_{b}:[N]\rightarrow \mathbb{R}_{\geq 0} $ by
	\begin{center}
		$  f_{b}(n):=\begin{cases}
			\dfrac{\varphi(W)}{W}\log p \quad if \  Wn+b=p, \ p\in A,  \\
			0  \qquad\qquad\qquad\qquad \  otherwise, 
			\end {cases}  $
		\end{center}
		
		\begin{center}
			$  \nu_{b}(n):=\begin{cases}
				\dfrac{\varphi(W)}{W}\log p \quad if \  Wn+b=p, \ p\in \mathbb{P},  \\
				0  \qquad\qquad\qquad\qquad \ otherwise. 
				\end {cases}  $	
			\end{center}
			
			Define function $ g:[W]\times \mathbb{N}\rightarrow \mathbb{R}_{\geq 0} $ by
			\begin{center}
				$ g(b,N):=\mathbb{E}_{n \in [N]}f_{b}(n). $
			\end{center}
			Obviously, $ f_{b}(n)\leq \nu_{b}(n) $  for all $ n \in [N] $. By the Siegel-Walfisz theorem, we know that $ \mathbb{E}_{n\in [N]}\nu_{b}(n)=1+o(1) $.
			
			Next, we will use a downset idea from \cite[Subsection 5.2]{Sal} to prove Theorem \ref{thm 1.2}.  Before we continue, let’s introduce some definitions. For $n\in \mathbb{N}$ and $a,b\in \mathbb{Z}_{n}$, we say $ a<b\  (\bmod\ n) $ if there are $ a^{\prime},b^{\prime}\in \{0,\ldots,n-1\} $ such that $ a^{\prime}\equiv a (\bmod\ n),\ b^{\prime}\equiv b (\bmod\ n) $ and $ a^{\prime}<b^{\prime} $. Let $q$ be a square-free positive integer. For $v\in\mathbb{Z}_{q}\cong \prod_{p|q}\mathbb{Z}_{p}$, define \begin{center}
				$D(v):=\{b\in \mathbb{Z}_{q}:\forall\ p\ |\ q,\ 0\leq b\leq v\ (\bmod\ p)  \}$.
			\end{center}
			For a set $A\subseteq \mathbb{Z}_{q}$, we say $A$ is a downset if $ D(v)\subseteq A$ for all $v\in A$. For $u\in \mathbb{Z}_{q}^{\ast} $ and $a\in \mathbb{Z}_{q}$, we say $u$ is an upper bound for $a$ if $ a<u\  (\bmod\ p)  $ for all $ p\ |\ q $. We say $u\in \mathbb{Z}_{q}^{\ast} $ is an upper bound for $A\subseteq \mathbb{Z}_{q}$ if $u$ is an upper bound for
			all elements of $A$.
			
			$ \mathit{Proof \ of \ Theorem \ \ref{thm 1.2}} $. Let $A\subseteq \mathbb{Z}_{q}^{\ast}$ with $|A|>\frac{1}{2}\varphi(q)$. Let  $u=q-1$. It is evident that 
			\begin{center}
				$ u\equiv p-1\ (\bmod\ p) $ 
			\end{center}
			for all $ p\ |\ q $.
			By \cite[Lemma 5.8]{Sal}, there exists a downset $A^{\prime}\subseteq \mathbb{Z}_{q} $ such that $|A|=|A^{\prime}|$, $u$ is an upper bound for $ A^{\prime} $ and $|4A^{\prime}|\leq|4A|$. Note that $2A^{\prime}$ is also a downset.
			
			Let $S=\{a\in \mathbb{Z}_{q}: \forall\ p\ |\ q,\ a<u \ (\bmod\ p) \}$ and $A^{\prime\prime}=A^{\prime}+\{1\} $. Clearly, $|S|=\varphi(q)$ and $A^{\prime}\subseteq S $. Note that $u-A^{\prime\prime}\subseteq S $. Therefore,
			\begin{center}
				$\begin{aligned}
					\#\{ u=a+b:a\in A^{\prime},b\in A^{\prime\prime} \}&=|A^{\prime}\cap (u-A^{\prime\prime})|=|A^{\prime}\setminus (S\setminus(u-A^{\prime\prime}))  |\\
					&\geq |A^{\prime}|-(|S|-|A^{\prime\prime}|)>0.
				\end{aligned}
				$
			\end{center}
			It follows that $u\in A^{\prime}+A^{\prime\prime}=2A^{\prime}+\{1\} $. Hence $u-1\in 2A^{\prime}$. Since $ 2A^{\prime} $ is a downset, we have $D(u-1)\subseteq 2A^{\prime} $. Note that $2D(u-1)=\{ a\in \mathbb{Z}_{q}:a\equiv 0\ (\bmod (2,q))  \} $. Therefore, 
			\begin{center}
				$ \{ a\in \mathbb{Z}_{q}:a\equiv 0\ (\bmod (2,q))  \}\subseteq 4A^{\prime} $.
			\end{center}
			This, together with $ |4A^{\prime}|\leq|4A| $ and 
			\begin{center}
				$ 4A\subseteq \{ a\in \mathbb{Z}_{q}:a\equiv 0\ (\bmod (2,q))  \} $,
			\end{center}
			yields
			$4A=\{ a\in \mathbb{Z}_{q}:a\equiv 0\ (\bmod (2,q))  \}$.\qed
			
			The following lemma is essentially a generalized version of Theorem \ref{thm 1.2}.
			
			\begin{lemma}\label{generalized version of thm 1.2}
				Let $ h:\mathbb{Z}_{W}^{\ast} \rightarrow [0,1) $ satisfy $ \mathbb{E}_{b \in \mathbb{Z}_{W}^{\ast}}h(b)>1/2 $. Then, for all $ n \in \mathbb{Z}_{W} $ with $  n \equiv 0 \ (\bmod \ 2) $, there exist $ b_{1},\ldots,b_{8}\in \mathbb{Z}_{W}^{\ast} $ such that $ n\equiv b_{1}+\cdots+b_{8} (\bmod\ W),\ h(b_{i})>0 $ for all $ i \in \{1,\ldots,8\} $ and 
				\begin{center}
					$ h(b_{1})+\cdots+h(b_{8})>4. $
				\end{center}
			\end{lemma}
			\begin{proof}
				The proof of Lemma \ref{generalized version of thm 1.2} is omitted, since it follows directly by
				repeating the arguments in the proof of \cite[Lemma 6.4]{Sal} with Theorem \ref{thm 1.2} in place of \cite[Proposition 5.2]{Sal}.
			\end{proof}
			
			The following lemma provides a lower bound for $ \mathbb{E}_{b\in \mathbb{Z}_{W}^{\ast}}g(b,N) $.
			
			\begin{lemma}\label{lower bound}
				Let $ \epsilon\in (0,1) $. Then \begin{center}
					$ \mathbb{E}_{b\in \mathbb{Z}_{W}^{\ast}}g(b,N)>(1-\epsilon)\delta_{A} $
				\end{center}
				provided that $ n_{0} $ is large enough.
			\end{lemma}
			\begin{proof}
				Note that $w\sim \log W$. When $p>2W $, we have $(p, W)=1$ provided that $ n_{0}$ is sufficiently large.
				Additionally, note that each prime $p\in A$	with $ 2W<p\leq WN+1 $ corresponds to a unique pair $(n,b)\in [N]\times[W]$ such that $ Wn+b=p $ and $ b\in \mathbb{Z}_{W}^{\ast} $. Therefore, 
				\begin{equation}\label{equ 1}
					 \begin{aligned}
						\mathbb{E}_{b \in \mathbb{Z}_{W}^{\ast}}g(b,N)
						&=\dfrac{1}{\varphi(W)N}\sum\limits_{\substack{b \in [W] \\ (b,W)=1}}\sum\limits_{\substack{n\in [N] \\ Wn+b=p\\ p \in A}}\dfrac{\varphi(W)}{W}\log p\\
						&\geq \dfrac{1}{WN}\sum\limits_{\substack{p \in A \\ 2W<p\leq WN+1 }}\log p. 
					\end{aligned} 
				\end{equation}
				Let $ X_{1}= 2W $ and $ X_{2}= WN+1  $. Define
				\begin{center}
					$\begin{aligned} A(t):=\sum\limits_{\substack{p\leq t \\ p \in A}}1. \end{aligned}$
				\end{center}
				Using partial summation, we have\begin{center}
					$\begin{aligned}
						\sum\limits_{\substack{p \in A \\ 2W<p\leq WN+1 }}\log p&=A(X_{2})\log X_{2}-A(X_{1})\log X_{1}-\int_{X_{1}}^{X_{2}}A(t)t^{-1}dt\\
						&\geq	(1-o(1))\delta_{A}X_{2}-(1+o(1))X_{1}-(1+o(1))\int_{X_{1}}^{X_{2}}\dfrac{1}{\log t}dt\\
						&=(1-o(1))\delta_{A}X_{2}.
					\end{aligned}$
				\end{center}
			This, together with (\ref{equ 1}), yields
			\begin{center}
				$ \mathbb{E}_{b\in \mathbb{Z}_{W}^{\ast}}g(b,N)>(1-\epsilon)\delta_{A} $.
			\end{center}	
				
			\end{proof}
			Finally, to apply the transference principle, we present a mean value result.
			
			\begin{proposition}\label{mean value estimate}
				Let $ \epsilon \in (0,1/6) $ and let $ N $ be sufficiently large in terms of $ \epsilon $. Let $ \delta_{A}>1/2+3\epsilon $. Then, for all $ n \in \mathbb{Z}_{W} $ with $ n \equiv 0 \ (\bmod \ 2) $, there exist $ b_{1},\ldots,b_{8} \in \mathbb{Z}_{W}^{\ast} $ such that $ n \equiv b_{1}+\cdots+b_{8}(\bmod \ W ), \ g(b_{i},N) > \epsilon/2$ for all $ i \in \left\lbrace 1,\ldots,8\right\rbrace  $ and 
				\begin{center}
					$ g(b_{1},N)+\cdots+g(b_{8},N) > 4(1+\epsilon) $.
				\end{center}
			\end{proposition}
			\begin{proof}
				For $ b\in \mathbb{Z}_{W}^{\ast} $, define\begin{center}
					$ h(b):=\max\bigg( 0, \dfrac{1}{1+\epsilon}\big( g(b,N)-\epsilon/2 \big) \bigg) $.
				\end{center}
				Since $ \mathbb{E}_{n\in [N]}\nu_{b}(n)=1+o(1) $, we know that $ h(b)\in [0,1) $ provided that $N$ is large enough in terms of $\epsilon$. On the other hand, by Lemma \ref{lower bound}, we have\begin{center}
					$ \mathbb{E}_{b\in \mathbb{Z}_{W}^{\ast}}h(b)\geq \dfrac{1}{1+\epsilon}\mathbb{E}_{b\in \mathbb{Z}_{W}^{\ast}}\big( g(b,N)-\epsilon/2 \big)>1/2 $.
				\end{center}
				Therefore, by Lemma \ref{generalized version of thm 1.2}, for all $ n \in \mathbb{Z}_{W} $ with $ n \equiv 0 \ (\bmod\ 2) $, there exist $ b_{1},\ldots,b_{8}\in \mathbb{Z}_{W}^{\ast} $ such that $ n\equiv b_{1}+\cdots+b_{8} \ (\bmod \ W), \ h(b_{i})>0  $ for all $ i\in \{1,\ldots,8\} $ and
				\begin{center}
					$ h(b_{1})+\cdots+h(b_{8})>4. $
				\end{center}
				By the definition of $h$, we know that $ g(b_{i},N)>\epsilon/2 $ for all $ i \in \{1,\ldots,8\} $ and 
				\begin{center}
					$ g(b_{1},N)+\cdots+g(b_{8},N)>4(1+\epsilon)$.
				\end{center}
			\end{proof}

			\section{Pseudorandomness}
			
			The Fourier transform of a finitely supported function $ f:\mathbb{Z}\rightarrow \mathbb{C} $ is defined by
			\begin{center}
				$\begin{aligned} \widehat{f}(\alpha)=\sum\limits_{n\in \mathbb{Z}}f(n)e(n \alpha) \end{aligned}$ .
			\end{center}
			For a function $ f:\mathbb{Z}\rightarrow \mathbb{C} $, the $ L^{r}$-norm is
			defined by 
			\begin{center}
				$\begin{aligned} ||f||_{r}=\bigg( \sum\limits_{n}|f(n)|^{r} \bigg)^{1/r}\end{aligned} $.
			\end{center}
			In this section, we use the circle method to prove the pseudorandomness of $ \nu_{b} $.
			
			\begin{proposition}\label{pseudorandomness}
				Let $ \alpha \in \mathbb{T} $ . For $ b\in [W] $ with $ (b,W)=1 $, we have
				\begin{center}
					$ \mid \widehat{\nu}_{b}(\alpha) - \widehat{1_{[N]}}(\alpha) \mid =o(N). $
				\end{center}
			\end{proposition}
			 Suppose that $ \sigma_{0} $ is a large positive constant  and  $ \sigma $ is a much larger positive constant in terms of $ \sigma_{0} $. Let $ L=\log(WN+W)$. For $ 1 \leq q \leq L^{\sigma} $ and $ 0\leq a\leq q-1 $ with $(a,q)=1$, write 
			 \begin{center}
			 	$ \mathfrak{M}(q,a):=\{\alpha \in \mathbb{T}:|\alpha-a/q|\leq L^{\sigma}(WN)^{-1} \} $.
			 \end{center} 
			 
			Let $ \mathfrak{M}$ be the union of all these sets $\mathfrak{M}(q,a) $. 
			Put $ \mathfrak{m}=\mathbb{T}\setminus\mathfrak{M} $. We call $ \mathfrak{M} $ major arcs and $ \mathfrak{m} $ minor arcs.
			For any $ b \in [W] $ with $ (b,W)=1 $, we have 
			\begin{equation}\label{approximation equation}
				\begin{aligned}
					\widehat{\nu_{b}}(\alpha)
					&=\dfrac{\varphi(W)}{W}\sum\limits_{\substack{ W+b \leq p\leq  WN+b  \\ p\equiv b (\bmod W)}}e\big(\alpha (p-b)W^{-1}\big)\log p\\
					&=\dfrac{\varphi(W)}{W}\sum\limits_{\substack{ p\leq  WN+b  \\ p\equiv b (\bmod W)}}e\big(\alpha (p-b)W^{-1}\big)\log p+O(W)\\
					&=\dfrac{\varphi(W)e(-\alpha b/W)}{W}\sum\limits_{\substack{p\leq Y \\ p \equiv b (\bmod W)}}e(\alpha p/W)\log p+O(W),
				\end{aligned}
			\end{equation}
			where $ Y= WN+b $. 
			
			We begin by providing an estimate on the minor arcs.
			
			\begin{lemma}\label{minor arcs}
				Let $ \alpha \in \mathfrak{m} $. Then 
				\begin{center}
					$\widehat{\nu}_{b}(\alpha)-\widehat{1_{[N]}}(\alpha)\ll NL^{-\sigma_{0}} $.
				\end{center}
			\end{lemma}
			\begin{proof}
				By Dirichlet's approximation theorem, there exist relatively prime integers $q$ and $a$ such that $1\leq q\leq WN/L^{\sigma} $ and $|\beta|\leq L^{\sigma}(qWN)^{-1} $, where $ \beta=\alpha-a/q $. Since $ \alpha\in [0,1) $ and $ L^{\sigma}=o(WN) $, we must have $ a/q\in [0,1) $.
				Therefore, since $\alpha \notin \mathfrak{M}$, we must have $ q>L^{\sigma} $. It follows that $ |\beta|<(WN)^{-1} $.
				Define 
				\begin{center}
					$\begin{aligned} S(t)=\sum\limits_{\substack{p\leq t \\ p \equiv b (\bmod W)}}e_{Wq}(ap) \end{aligned} $
				\end{center}
				and 
				\begin{center}
					$f(t)=e(\beta t/W)\log t.$
				\end{center}
				When $ t\leq Y $, we have $ f^{\prime}(t)\ll (W^{2}N)^{-1}L $. Therefore, we have
				\begin{center}
					$ \begin{aligned}
						\int_{1}^{YL^{-2\sigma_{0}}}S(t)f'(t)dt\ll (YL^{-2\sigma_{0}})^{2}(W^{2}N)^{-1}L \ll NL^{-4\sigma_{0}+1}.
					\end{aligned} $
				\end{center}
				By partial summation, we have 
				\begin{center}
					$\begin{aligned}
						\sum\limits_{\substack{p\leq Y \\ p \equiv b (\bmod W)}}e(\alpha p/W)\log p&= \sum\limits_{\substack{p\leq Y \\ p \equiv b (\bmod W)}}e(\beta p/W)e_{Wq} (ap)\log p\\
						&=S(Y)f(Y)-\int_{1}^{Y}S(t)f'(t)dt\\
						&=S(Y)f(Y)-\int_{YL^{-2\sigma_{0}}}^{Y}S(t)f'(t)dt+O(NL^{-4\sigma_{0}+1}).
					\end{aligned}$
				\end{center}
				It follows from (\ref{approximation equation}) that \begin{equation}\label{minor arc 1}
					\widehat{\nu_{b}}(\alpha)\ll NL^{-4\sigma_{0}+1}+L\sup_{YL^{-2\sigma_{0}}\leq t\leq Y}|S(t)|. 
				\end{equation}
				As in  \cite[proof of Lemma 3.1]{Chow}, we use \cite[Theorem 10]{Hua} to obtain
				\begin{equation}\label{minor arc 2}
					S(t)\ll YL^{-\sigma_{0}-1}W^{-1} 
				\end{equation}
				for all $ YL^{-2\sigma_{0}}\leq t\leq Y $.
				By using Dirichlet's approximation theorem again, we know that there exist relatively prime integers $q$ and $a$ such that $1\leq q\leq L^{\sigma} $ and $|\alpha-a/q|\leq (qL^{\sigma})^{-1} $. Similarly, we have $ a/q\in [0,1) $. Note that $ \alpha \notin \mathfrak{M} $, we must have $ |\alpha-a/q|> L^{\sigma}(WN)^{-1} $. Obviously, in this case, we have $ ||q\alpha||=|q\alpha-a| $. Therefore,
				\begin{center}
					$ \widehat{1_{[N]}}(\alpha)\ll ||\alpha||^{-1}\leq \dfrac{q}{||q\alpha||}\leq \dfrac{WN}{L^{\sigma}}\ll NL^{-\sigma+1} $.
				\end{center}
				This, together with  (\ref{minor arc 1}) and (\ref{minor arc 2}), yields
				 \begin{center}
					$ \widehat{\nu}_{b}(\alpha)-\widehat{1_{[N]}}(\alpha)\ll NL^{-\sigma_{0}} $.
				\end{center}

			\end{proof}
			
			For $ b\in [W] $ with $(b,W)=1$, define
			\begin{center}
				$\begin{aligned}
					S_{q}^{\ast}(a,b)=\sum\limits_{\substack{r=0 \\ (b+Wr,q)=1}}^{q-1}e(ar/q)
				\end{aligned}$.
			\end{center}
			Additionally, define
			\begin{center}
				$ \begin{aligned}
					I(\beta)=\int_{0}^{N}e(\beta t)dt.
				\end{aligned}  $
			\end{center}
			
			Next, we provide the asymptotic formula for $ \widehat{\nu_{b}}(\alpha) $ on the major arcs.
			
			\begin{lemma}\label{asymptotic formula on major arcs}
				Let $ \alpha \in \mathfrak{M}(q,a) $ and $ \beta =\alpha-a/q \in [-L^{\sigma}(WN)^{-1},L^{\sigma}(WN)^{-1}] $. Then
				\begin{center}
					$\widehat{\nu_{b}}(\alpha)=\dfrac{\varphi(W)}{\varphi(Wq)}S_{q}^{\ast}(a,b)I(\beta)+O(Ne^{-C_{2}\sqrt{L}}) $.
				\end{center}
			\end{lemma}
			\begin{proof}
				With $ f(t)=e(\beta t/W)\log t $ and $ S(t)=\sum\limits_{\substack{p\leq t \\ p \equiv b (\bmod W)}}e_{Wq}(ap)  $, we have 
				\begin{equation}\label{major 1}
					\begin{aligned}  \sum\limits_{\substack{p\leq Y \\ p \equiv b (\bmod W)}}e(\alpha p/W)\log p
						&=\sum\limits_{n=2}^{Y}(S(n)-S(n-1))f(n)\\
						&=S(Y)f(Y+1)+\sum\limits_{n=2}^{Y}S(n)(f(n)-f(n+1)).
					\end{aligned}  
				\end{equation}
				Let
				\begin{center}
					$\begin{aligned} V_{q}(a,b)=\sum\limits_{\substack{r=0 \\ (b+Wr,q)=1}}^{q-1}e_{Wq}(a(b+Wr)).\end{aligned} $
				\end{center}
				
				Note that $ Wq\leq L^{\sigma+1} $ and $ n\leq W(N+1) $. We use the Siegel-Walfisz theorem \cite[Lemma 7.14]{Hua} to obtain
				\begin{equation}\label{major 2}
					\begin{aligned}
						S(n)
						&=\sum\limits_{\substack{r=0 \\ (b+Wr,Wq)=1}}^{q-1}e_{Wq}(a(b+Wr))\sum\limits_{\substack{p\leq n\\ p \equiv b+Wr (\bmod Wq)}}1+O(Wq)\\
						&=\dfrac{\mathtt{Li}(n)}{\varphi(Wq)}V_{q}(a,b)+O\big((WN+W)e^{-C_{1}\sqrt{L}}\big).
					\end{aligned} 
				\end{equation}
				Note that 	$f(n)-f(n+1)\ll (W^{2}N)^{-1}L^{\sigma+1}$ for 
				$ n\leq Y $. Substituting (\ref{major 2}) into (\ref{major 1}), we have
				\begin{center}
					$\begin{aligned}
						\sum\limits_{\substack{p\leq Y \\ p \equiv b (\bmod W)}}e(\alpha p/W)\log p&=\dfrac{V_{q}(a,b)}{\varphi(Wq)}\bigg[\mathtt{Li}(Y)f(Y+1)+\sum\limits_{n=2}^{Y}\mathtt{Li}(n)(f(n)-f(n+1))  \bigg]\\
						&\ \ \ +O(Ne^{-C_{2}\sqrt{L}})\\
						&=\dfrac{V_{q}(a,b)}{\varphi(Wq)}\bigg[\mathtt{Li}(Y)f(Y+1)+\sum\limits_{n=3}^{Y}\mathtt{Li}(n)(f(n)-f(n+1))  \bigg]\\
						&\ \ \ +O(Ne^{-C_{2}\sqrt{L}})\\
						&=\dfrac{V_{q}(a,b)}{\varphi(Wq)}\sum\limits_{n=3}^{Y}\int_{n-1}^{n}\dfrac{f(n)}{\log t}dt+O(Ne^{-C_{2}\sqrt{L}}).
					\end{aligned}$
				\end{center}
				
				Note that $ f(n)=f(t)+O(W^{-2}N^{-1}L^{\sigma+1}) $ for $  n-1\leq t\leq n $.
				Therefore, we have 
				\begin{equation}\label{major 3}
					\begin{aligned}
						\sum\limits_{\substack{p\leq Y \\ p \equiv b (\bmod W)}}e(\alpha p/W)\log p&=\dfrac{V_{q}(a,b)}{\varphi(Wq)}
						\int_{2}^{Y}e(\beta t/W)dt+O(Ne^{-C_{2}\sqrt{L}})\\
						&=\dfrac{W}{\varphi(Wq)}V_{q}(a,b)I(\beta)+O(Ne^{-C_{2}\sqrt{L}}). 
					\end{aligned}
				\end{equation}
				Substituting (\ref{major 3}) into (\ref{approximation equation}), we have 
				\begin{center}
					$\begin{aligned} \widehat{\nu_{b}}(\alpha)&=\dfrac{\varphi(W)e(-\alpha b/W)}{\varphi(Wq)}V_{q}(a,b)I(\beta)+O(Ne^{-C_{2}\sqrt{L}})\\
						&=\dfrac{\varphi(W)e(-\beta b/W)}{\varphi(Wq)}S_{q}^{\ast}(a,b)I(\beta)+O(Ne^{-C_{2}\sqrt{L}}).
					\end{aligned}$
				\end{center}
				This, together with the estimate
				\begin{center}
					$ \begin{aligned}
						e(-\beta b/W)I(\beta)=\int_{0}^{N}e(\beta(t-b/W))dt=I(\beta)+O(1), \end{aligned}$
				\end{center}
				yields Lemma \ref{asymptotic formula on major arcs}.
			\end{proof}
			
			When $\alpha\in \mathfrak{M}(1,0) $, it follows from Lemma \ref{asymptotic formula on major arcs} that \begin{center}
				$ \widehat{\nu_{b}}(\alpha)=I(\alpha)+O(Ne^{-C_{2}\sqrt{L}}) $.
			\end{center}
			As in \cite[Section 4]{Chow}, we use the Euler-Maclaurin summation to obtain 
			\begin{center}
				$ \widehat{1_{[N]}}(\alpha)=I(\alpha)+O(L^{\sigma}) $.
			\end{center}
			Therefore, when $\alpha\in \mathfrak{M}(1,0) $, we have 
			\begin{equation}\label{major arc q=1}
				\widehat{\nu}_{b}(\alpha)-\widehat{1_{[N]}}(\alpha)\ll Ne^{-C_{2}\sqrt{L}}  .
			\end{equation}
			
			Finally, for $2\leq q\leq L^{\sigma}$, let $ \alpha \in \mathfrak{M}(q,a) $  and $ \beta =\alpha-a/q \in [-L^{\sigma}(WN)^{-1},\ L^{\sigma}(WN)^{-1}] $.
			
			Let $ q=uv $, where $ u $ is $ w$-$smooth $ and $ (v,W)=1 $. Note that $ (u,v)=1 $. There are $ \bar{u},\ \bar{v}\in \mathbb{Z} $ such that $ u\bar{u}+v\bar{v}=1 $. Therefore, we have 
			\begin{center}
				$\begin{aligned}
					S_{q}^{\ast}(a,b)
					&=\sum\limits_{\substack{r_{1}=0 \\ (b+Wr_{1}v,u)=1}}^{u-1}\sum\limits_{\substack{r_{2}=0 \\ (b+Wr_{2}u,v)=1}}^{v-1}e_{uv}\big(a(u\bar{u}+v\bar{v})(r_{1}v+r_{2}u) \big)\\
					&=\sum\limits_{\substack{r_{1}=0 \\ (b+Wr_{1}v,u)=1}}^{u-1}e_{u}(a\bar{v}r_{1}v) \sum\limits_{\substack{r_{2}=0 \\ (b+Wr_{2}u,v)=1}}^{v-1}e_{v}(a\bar{u}r_{2}u)\\
					&=\sum\limits_{\substack{r_{1}=0 \\ (b+Wr_{1},u)=1}}^{u-1}e_{u}(a_{1}r_{1})\sum\limits_{\substack{r_{2}=0 \\ (b+Wr_{2},v)=1}}^{v-1}e_{v}(a_{2}r_{2}),
				\end{aligned}$
			\end{center}
			where $ a_{1}\equiv a\bar{v} \ (\bmod \ u)  $	and $ a_{2}\equiv a\bar{u} \ (\bmod \ v) $.
			Therefore, we have
			\begin{equation}\label{decomposition}
				S_{q}^{\ast}(a,b)=S_{u}^{\ast}(a_{1},b)S_{v}^{\ast}(a_{2},b).
			\end{equation}	
			Clearly, the coprimality condition in the expression for $ S_{u}^{\ast}(a_{1},b) $ can be removed.
			Note that $ (a_{1},u)=1 $. Therefore,  \begin{equation}\label{decomposition 1}
				\begin{aligned}
					S_{u}^{\ast}(a_{1},b)=\sum\limits_{r_{1}=0}^{u-1}e(a_{1}r_{1}/u)=\begin{cases}
						1 \quad if \  u=1,   \\
						0  \quad  if\ u>1. 
						\end {cases}  \end{aligned}  
				\end{equation}
				Since $(v,W)=1$, let $\overline{W}\in \mathbb{Z}$ be such that $W\overline{W}\equiv 1\ (\bmod \ v) $. By substituting the variable $ t=b\overline{W}+r $, we have 
				\begin{center}
					$\begin{aligned}
						e_{Wv}(a_{2}b)S_{v}^{\ast}(a_{2},b)=\sum\limits_{\substack{r=0 \\ (b+Wr,v)=1}}^{v-1}e_{v}(a_{2}(b+Wr)/W)=	\sum\limits_{\substack{t=0 \\ (t,v)=1}}^{v-1}e_{v}(a_{2}t).
					\end{aligned}$
				\end{center}
				Note that $(a_{2},v)=1$. Therefore, by \cite[Lemma 8.5]{Hua}, we know that \begin{equation}\label{decomposition 2}
					S_{v}^{\ast}(a_{2},b)\ll v^{{1/2}+\epsilon} .
				\end{equation}

				\begin{lemma}\label{major arc case}
					For $ 2\leq q\leq L^{\sigma} $, let $ \alpha \in \mathfrak{M}(q,a) $.
					Then 
					\begin{center}
						$ \widehat{\nu_{b}}(\alpha)-\widehat{1_{[N]}}(\alpha)\ll_{\epsilon} w^{\epsilon-1/2}N $
					\end{center}
					for any $ \epsilon>0 $.
				\end{lemma}
				\begin{proof}
					First, we prove that $ \widehat{\nu_{b}}(\alpha)\ll_{\epsilon} w^{\epsilon-1/2}N $	.
					When $ u>1 $, it follws from (\ref{decomposition}), (\ref{decomposition 1}) and Lemma \ref{asymptotic formula on major arcs} that 
					\begin{center}
						$ \widehat{\nu_{b}}(\alpha)\ll Ne^{-C_{2}\sqrt{L}} $.
					\end{center} 	
					Therefore, we assume that $u=1$. In this case, we have $q=v$. Note that $ (v,W)=1 $. Therefore, we have $q>w$. By (\ref{decomposition}) , (\ref{decomposition 2}) and Lemma \ref{asymptotic formula on major arcs}, we have 
					\begin{center}
						$\begin{aligned}
							\widehat{\nu_{b}}(\alpha)&\ll\dfrac{\varphi(W)}{\varphi(Wq)}N|S_{q}^{\ast}(a,b)|+Ne^{-C_{2}\sqrt{L}}\\	&\ll \dfrac{N}{\varphi(q)}q^{1/2+\epsilon}+Ne^{-C_{2}\sqrt{L}}\\
							&\ll Nq^{-1/2+2\epsilon}+Ne^{-C_{2}\sqrt{L}}\\
							&\ll w^{2\epsilon-1/2}N.
						\end{aligned} 
						$
					\end{center}
					Next, we estimate  $ \widehat{1_{[N]}}(\alpha) $. Since \begin{center}
						$ ||\alpha||=||a/q+\beta||\geq ||a/q||-||\beta||\geq q^{-1}-|\beta|\geq q^{-1}-L^{\sigma}/(WN)\gg L^{-\sigma} $,
					\end{center} 
					we have \begin{center}
						$ \widehat{1_{[N]}}(\alpha)\ll ||\alpha||^{-1}\ll L^{\sigma} $.
					\end{center}
					This, together with our estimate for $ \widehat{\nu_{b}}(\alpha) $, yields that 
					\begin{center}
						$ \widehat{\nu_{b}}(\alpha)-\widehat{1_{[N]}}(\alpha)\ll_{\epsilon} w^{\epsilon-1/2}N $.
					\end{center}

				\end{proof}
				
				Combining Lemma \ref{minor arcs}, Lemma \ref{major arc case} and (\ref{major arc q=1}), we obtain Proposition \ref{pseudorandomness} .
				
				\section{The restriction estimate}
				In this section, we will establish the restriction estimate for $ f_{b} $.
				
				\begin{proposition}\label{restriction estimate}
					Let $ \rho>4 $. For $ b\in [W] $ with $ (b,W)=1 $, we have
					\begin{center}
						$\begin{aligned}
							\int_{\mathbb{T}}|\widehat{f_{b}}(\alpha)|^{\rho}d\alpha\ll_{\rho} N^{\rho-1}.
						\end{aligned}  $
					\end{center}
				\end{proposition}
				
				Before proving Proposition \ref{restriction estimate}, we first quickly establish the following results.
				
				\begin{lemma}\label{lem 5.2}
					For $ q\leq L^{\sigma} $, let $ \alpha\in \mathfrak{M}(q,a) $. Then \begin{center}
						$ \widehat{\nu_{b}}(\alpha)\ll q^{\epsilon-1/2}\min\{N,|\alpha-a/q|^{-1}\}$.
					\end{center}
				\end{lemma}
				\begin{proof}
					When \( q = 1 \), it is evident that \( 	S_{q}^{\ast}(a,b)  = 1 \).
					
					For $ q>2 $, recall the decomposition of \( q \) in Section 4, where we let \( q = uv \), with \( u \) being \( w \)-smooth and \( (v, W) = 1 \). When \( u = 1 \) and \( u > 1 \), we obtain \( S_{q}^{\ast}(a,b)\ll q^{1/2+\epsilon} \) and \( S_{q}^{\ast}(a,b)=0 \) respectively from (\ref{decomposition}), (\ref{decomposition 1}), (\ref{decomposition 2}). Therefore, \begin{center}
						$ S_{q}^{\ast}(a,b)\ll q^{1/2+\epsilon} $
					\end{center}
					for all $ 1\leq q\leq L^{\sigma} $.
					This, together with Lemma \ref{asymptotic formula on major arcs} and the estimate 
					\begin{center}
						$ I(\beta)\ll \min\{N,\ |\beta|^{-1}\}=\min\{N,\ |\alpha-a/q|^{-1} \} $,
					\end{center}
					yields
					\begin{center}
						$ \begin{aligned}
							\widehat{\nu_{b}}(\alpha)&\ll \varphi(q)^{-1}q^{1/2+\epsilon}\min\{N,\ |\alpha-a/q|^{-1} \}+Ne^{-C_{2}\sqrt{L}}\\
							&\ll  q^{2\epsilon-1/2}\min\{N,|\alpha-a/q|^{-1}\}+Ne^{-C_{2}\sqrt{L}}.
						\end{aligned}
						$
					\end{center}
					By $ q\leq L^{\sigma} $ and $ |\alpha-a/q|\leq L^{\sigma}(WN)^{-1} $, we have 
					\begin{center}
						$ Ne^{-C_{2}\sqrt{L}}\ll q^{2\epsilon-1/2}N $
					\end{center}
					and
					\begin{center}
						$ Ne^{-C_{2}\sqrt{L}}\ll q^{2\epsilon-1/2}|\alpha-a/q|^{-1} $.
					\end{center}
					Therefore, 
					\begin{center}
						$ \widehat{\nu_{b}}(\alpha)\ll q^{2\epsilon-1/2}\min\{N,|\alpha-a/q|^{-1}\} $.
					\end{center}
				\end{proof}
				
				\begin{lemma}\label{lem 5.3}
					For $ b\in [W] $ with $ (b,W)=1 $, we have
					\begin{center}
						$\begin{aligned}
							\int_{\mathbb{T}}|\widehat{f_{b}}(\alpha)|^{4}d\alpha \ll N^{3}L^{2}.
						\end{aligned}  $
					\end{center}
				\end{lemma}
				\begin{proof}
					Recall that $ L=\log(WN+W) $. We have
					\begin{center}
						$\begin{aligned}
							\int_{\mathbb{T}}|\widehat{f_{b}}(\alpha)|^{2}d\alpha &=\int_{\mathbb{T}}\sum\limits_{n_{1}\in[N]}f_{b}(n_{1})e(n_{1}\alpha)\sum\limits_{n_{2}\in[N]}f_{b}(n_{2})e(-n_{2}\alpha)d\alpha\\
							&=\sum\limits_{n_{1},n_{2}\in[N]}f_{b}(n_{1})f_{b}(n_{2})\int_{\mathbb{T}}e((n_{1}-n_{2})\alpha)d\alpha\\
							&=\sum\limits_{n\in[N]}f_{b}(n)^{2}\\
							&\ll NL^{2}.
						\end{aligned}
						$
					\end{center}
					Since $ ||\widehat{f_{b}}||_{\infty}\leq ||f_{b}||_{1}\leq ||\nu_{b}||_{1}=(1+o(1)) N $, we have 
					\begin{center}
						$ \begin{aligned}
							\int_{\mathbb{T}}|\widehat{f_{b}}(\alpha)|^{4}d\alpha\ll (||\widehat{f_{b}}||_{\infty})^{2}\int_{\mathbb{T}}|\widehat{f_{b}}(\alpha)|^{2}d\alpha\ll N^{3}L^{2}.
						\end{aligned} $
					\end{center}
				\end{proof}
				$ \mathit{Proof \ of \ Proposition \ \ref{restriction estimate}.} $ For $ \delta\in (0,1) $, define\begin{center}
					$ \mathcal{R}_{\delta} =\{\alpha\in \mathbb{T}:|\widehat{f_{b}}(\alpha)|>\delta N \} $.
				\end{center}
				Let $ \epsilon_{0}\in (0,\rho-4) $. It suffices to show that 
				\begin{equation}\label{equ 5.1}
					meas(\mathcal{R}_{\delta})\ll_{\epsilon_{0}}\dfrac{1}{\delta^{4+\epsilon_{0}}N} .
				\end{equation}
				The fact is, when (\ref{equ 5.1}) holds, we have
				\begin{center}
					$ \begin{aligned}
						\int_{\mathbb{T}}|\widehat{f_{b}}(\alpha)|^{\rho}d\alpha
						&\leq \sum\limits_{j\geq 0}\bigg(\dfrac{N}{2^{j-1}}\bigg)^{\rho}meas\bigg\{\alpha\in \mathbb{T}:\dfrac{N}{2^{j}}<|\widehat{f_{b}}(\alpha)|\leq\dfrac{N}{2^{j-1}}\bigg\}\\
						&\ll N^{\rho-1}\sum\limits_{j\geq 0}2^{j(4+\epsilon_{0}-\rho)}\ll N^{\rho-1}.
					\end{aligned} $
				\end{center}
				When $ \delta\leq L^{-2/\epsilon_{0}} $, from Lemma \ref{lem 5.3}, we have 
				\begin{center}
					$\begin{aligned}
						meas(\mathcal{R}_{\delta})\leq (\delta N)^{-4} \int_{\mathbb{T}}|\widehat{f_{b}}(\alpha)|^{4}d\alpha \ll \delta^{-4-\epsilon_{0}}N^{-1}.
					\end{aligned}  $
				\end{center}
				Therefore, we assume $ L^{-2/\epsilon_{0}}<\delta<1 $. Suppose that $\theta_{1},\ldots,\theta_{R}$ are $ N^{-1} $-spaced points in $\mathcal{R}_{\delta}$. It suffices to show that
				\begin{equation}\label{equ 5.3}
					R\ll \delta^{-4-\epsilon_{0}} .
				\end{equation}
				Let $ \iota=2+\epsilon_{0}/3 $. For $ n \in [N] $, let $ a_{n}\in \{0,1\} $ be such that $ f_{b}(n)=a_{n}\nu_{b}(n) $. For $ r \in [R] $, let $ b_{r}\in \mathbb{C} $ be such that $ b_{r}\widehat{f_{b}}(\theta_{r})=|\widehat{f_{b}}(\theta_{r})| $. It is evident that we can assume $ |b_{r}|=1 $ for all $ r \in [R] $. By applying the Cauchy-Schwarz inequality, we have
				\begin{center}
					$ \begin{aligned}
						(\delta NR)^{2}&\leq \bigg( \sum\limits_{r=1}^{R}|\widehat{f_{b}}(\theta_{r})| \bigg)^{2}=\bigg(\sum\limits_{r=1}^{R}b_{r}\sum\limits_{n\in [N]}a_{n}\nu_{b}(n)e(n\theta_{r}) \bigg)^{2}\\
						&\leq \bigg(\sum\limits_{n\in[N]}|a_{n}|^{2}\nu_{b}(n)\bigg)\bigg(\sum\limits_{n\in [N]}\nu_{b}(n)\bigg|\sum\limits_{r=1}^{R}b_{r}e(n\theta_{r})\bigg|^{2}\bigg)\\
						&\ll N\sum\limits_{1\leq r,r^{\prime}\leq R}|\widehat{\nu_{b}}(\theta_{r}-\theta_{r^{\prime}})|. 
					\end{aligned} $
				\end{center}
				On the other hand, by Hölder's inequality, we have
				\begin{center}
					$ \begin{aligned}
						\sum\limits_{1\leq r,r^{\prime}\leq R}|\widehat{\nu_{b}}(\theta_{r}-\theta_{r^{\prime}})|\leq 
						\bigg(\sum\limits_{1\leq r,r^{\prime}\leq R}|\widehat{\nu_{b}}(\theta_{r}-\theta_{r^{\prime}})|^{\iota} \bigg)^{1/\iota}R^{2(1-1/\iota)}.
					\end{aligned} $
				\end{center}
				Therefore, 
				\begin{equation}\label{equ 5.2}
					\begin{aligned} \delta^{2\iota}N^{\iota}R^{2}\ll \sum\limits_{1\leq r,r^{\prime}\leq R}|\widehat{\nu_{b}}(\theta_{r}-\theta_{r^{\prime}})|^{\iota}. \end{aligned} 
				\end{equation}
				By (\ref{minor arc 1}), (\ref{minor arc 2}) and the range of $L$, we know that
				\begin{center}
					$ \begin{aligned}
						\sum\limits_{\substack{1\leq r,r^{\prime}\leq R\\\theta_{r}-\theta_{r^{\prime}}\in \mathfrak{m}} }|\widehat{\nu_{b}}(\theta_{r}-\theta_{r^{\prime}})|^{\iota}\ll R^{2}N^{\iota}L^{-\sigma_{0}\iota}=o(\delta^{2\iota}N^{\iota}R^{2}).
					\end{aligned}$
				\end{center}
				For $ 1\leq q\leq L^{\sigma} $, write $ \mathfrak{M}(q):=\bigcup\limits^{q-1}_{\substack{a=0  \\ (a,q)=1}}\mathfrak{M}(q,a) $. Let $ Q=C+\delta^{-5} $, where $ C $ is a large positive constant. By Lemma \ref{lem 5.2}, we have 
				\begin{center}
					$ \begin{aligned}
						\sum\limits_{\substack{1\leq r,r^{\prime}\leq R\\ \theta_{r}-\theta_{r^{\prime}}\in \mathfrak{M}(q),\ q>Q}}|\widehat{\nu_{b}}(\theta_{r}-\theta_{r^{\prime}})|^{\iota}\ll R^{2}N^{\iota}Q^{\epsilon-\iota/2}.
					\end{aligned} $
				\end{center}
				Note that $C$ is sufficiently large. The right-hand side of the above inequality is negligible compared to the left-hand side of (\ref{equ 5.2}). Therefore, by Lemma \ref{lem 5.2}, we have 
				\begin{center}
					$ \begin{aligned}
						\delta^{2\iota}N^{\iota}R^{2}&\ll \sum\limits_{q\leq Q}\sum\limits_{\substack{0\leq a\leq q-1\\ (a, q)=1}}\sum\limits_{\substack{1\leq r,r^{\prime}\leq R\\ \theta_{r}-\theta_{r^{\prime}}\in \mathfrak{M}(q,a)}}|\widehat{\nu_{b}}(\theta_{r}-\theta_{r^{\prime}})|^{\iota}\\
						&\ll \sum\limits_{q\leq Q}\sum\limits_{\substack{0\leq a\leq q-1\\ (a, q)=1}}\sum\limits_{\substack{1\leq r,r^{\prime}\leq R\\ \theta_{r}-\theta_{r^{\prime}}\in \mathfrak{M}(q,a)}}N^{\iota}q^{\epsilon-\iota/2}\big( \min\{1,( N|\theta_{r}-\theta_{r^{\prime}}-a/q| )^{-1}\} \big)^{\iota}.
					\end{aligned} $
				\end{center}
				It follows that
				\begin{center}
					$ \begin{aligned}
						\delta^{2\iota}R^{2}\ll \sum\limits_{q\leq Q}\sum\limits_{\substack{0\leq a\leq q-1\\ (a, q)=1}}\sum\limits_{\substack{1\leq r,r^{\prime}\leq R\\ \theta_{r}-\theta_{r^{\prime}}\in \mathfrak{M}(q,a)}}q^{\epsilon-\iota/2}\big( \min\{1,( N|\theta_{r}-\theta_{r^{\prime}}-a/q| )^{-1}\} \big)^{\iota}.
					\end{aligned} $
				\end{center}
				When $ 1\leq ( N|\theta_{r}-\theta_{r^{\prime}}-a/q| )^{-1} $, we have
				\begin{center}
					$ 1\ll (1+N|\theta_{r}-\theta_{r^{\prime}}-a/q| )^{-1} $.
				\end{center}
				When $ 1> ( N|\theta_{r}-\theta_{r^{\prime}}-a/q| )^{-1} $, we have
				\begin{center}
					$ ( N|\theta_{r}-\theta_{r^{\prime}}-a/q| )^{-1} \ll (1+N|\theta_{r}-\theta_{r^{\prime}}-a/q| )^{-1} $.
				\end{center}
				Therefore,
				\begin{center}
					$\begin{aligned}
						\delta^{2\iota}R^{2}\ll \sum\limits_{q\leq Q}\sum\limits_{\substack{0\leq a\leq q-1\\ (a, q)=1}}\sum\limits_{\substack{1\leq r,r^{\prime}\leq R\\ \theta_{r}-\theta_{r^{\prime}}\in \mathfrak{M}(q,a)}}\dfrac{q^{\epsilon-\iota/2}}{(1+N|\theta_{r}-\theta_{r^{\prime}}-a/q|)^{\iota}}.
					\end{aligned}  $
				\end{center}
				Let
				\begin{center}
					$ \begin{aligned}
						G(\alpha)=\sum\limits_{q\leq Q}\sum\limits_{0\leq a\leq q-1}\dfrac{q^{\epsilon-\iota/2}}{(1+N|\sin(\alpha-a/q)|)^{\iota}}.
					\end{aligned} $
				\end{center}
				We have
				\begin{center}
					$\begin{aligned}
						\delta^{2\iota}R^{2}\ll \sum\limits_{1\leq r,r^{\prime}\leq R}G(\theta_{r}-\theta_{r^{\prime}}).
					\end{aligned}  $
				\end{center}

				Then following the argument of \cite[Eq.(4.16)]{Bou} but with $ N^{2} $ replaced by $ N $, we obtain the desired bound (\ref{equ 5.3}), thereby concluding the proof of Proposition \ref{restriction estimate}.

				\section{Proof of Theorem 1.1}
				We now apply Salmensuu's transference lemma \cite[Proposition 3.9]{Sal} together with Proposition \ref{mean value estimate}, Proposition \ref{pseudorandomness}, Proposition \ref{restriction estimate} to prove Theorem \ref{thm 1.1}.
				
				$ \mathit{Proof \ of \ Theorem \ \ref{thm 1.1}.} $ Recall that $ n_{0} $ is a sufficiently large even integer. Our goal is to prove that $ n_{0}\in 8A $.
				
				Let $ \epsilon\in (0,1/6) $ be such that $ \delta_{A}>1/2+3\epsilon $. By Proposition \ref{mean value estimate}, we know that there are $ b_{1},\ldots,b_{8}\in [W] $ with $ (b_{i},W)=1 $ for all $ i\in\{1,\ldots,8 \} $ such that $ n_{0}\equiv b_{1}+\cdots+b_{8}\ (\bmod\ W ) $,
				\begin{center}
					$ \mathbb{E}_{n\in[N]} f_{b_{1}}(n)+\cdots+f_{b_{8}}(n)>4(1+\epsilon) $
				\end{center}
				and
				\begin{center}
					$ \mathbb{E}_{n\in[N]} f_{b_{i}}(n)>\epsilon/2 $
				\end{center}
				for all $ i\in\{1,\ldots,8 \} $. Thus, we have obtained the mean condition required for the use of \cite[Proposition 3.9]{Sal}. By Proposition \ref{pseudorandomness}, we know that $ f_{b_{1}},\ldots,f_{b_{8}} $ are $ \eta$-$\mathit{pseudorandom}$ for any $\eta >0$ (see \cite[Definition 3.4]{Sal} for the precise definition of $ \eta$-$\mathit{pseudorandom}$). By Proposition \ref{restriction estimate}, we have for any $ q\in (7,8) $ that $ f_{b_{1}},\ldots,f_{b_{8}} $ are $q$-$\mathit{restricted\ with\ constant}\ K $ for some $ K\geq 1 $ (see \cite[Definition 3.5]{Sal} for the precise definition of $q$-$\mathit{restricted\ with\ constant}\ K $). Note that as $ n_{0} $ tends to infinity, $ N $ also approaches infinity. Let $ \eta $ be sufficiently small in terms of $ \epsilon,K $ and $ q $. By \cite[Proposition 3.9]{Sal}, we know that
				\begin{center}
					$\begin{aligned}
						\sum\limits_{n=n_{1}+\cdots+n_{8}} f_{b_{1}}(n_{1})\cdots f_{b_{8}}(n_{8})>0 
					\end{aligned}
					$
				\end{center}  
				for all $ n\in \big( 4(1-\kappa^{2})N,4(1+\kappa)N \big) $, where $ \kappa=\epsilon/32 $. It follows that
				\begin{center}
					$ Wn+b_{1}+\cdots+b_{8}\in 8A $
				\end{center}
				for all $ n\in \big( 4(1-\kappa^{2})N,4(1+\kappa)N \big) $. Let $ n^{\prime}=(n_{0}-b_{1}-\cdots-b_{8})/W $. Then $ n^{\prime}\in\mathbb{Z} $ and $ n^{\prime}\sim n_{0}/W \sim 4N $. Therefore, $ n^{\prime}\in \big( 4(1-\kappa^{2})N,4(1+\kappa)N \big) $ provided that $ n_{0} $ is sufficiently large in terms of $ \epsilon $. It follows that $ n_{0}\in 8A $.\qed

\end{document}